\documentclass{article}

\usepackage{amsfonts,amsmath, amsthm, amssymb}
\usepackage[latin1]{inputenc}
\usepackage[english]{babel}
\usepackage{graphicx}
\usepackage{enumerate}

\def\td{tree-decom\-po\-si\-tion}

\newcommand{\G}{\mathcal G}

\newcommand{\Po}{\mathcal P}
\newcommand{\Q}{\mathcal Q}
\newcommand{\V}{\mathcal V}

\newcommand{\adm}{\rm adm}
\newcommand{\tw}{\rm tw}

\newcommand{\sub}{\subseteq}

\newcommand{\infadm}{\mbox{$\infty$-admissibility}}

\newcommand{\comment}[1]{}

\newtheorem{theorem}{Theorem}
\newtheorem{lemma}[theorem]{Lemma}
\newtheorem{corollary}[theorem]{Corollary}

\theoremstyle{definition}

\def\?#1{\vadjust{\vbox to 0pt{\vss\vskip-8pt\leftline{%
     \llap{\hbox{\vbox{\pretolerance=-1
     \doublehyphendemerits=0\finalhyphendemerits=0
     \hsize16truemm\tolerance=10000\small
     \lineskip=0pt\lineskiplimit=0pt
     \rightskip=0pt plus16truemm\baselineskip8pt\noindent
     \hskip0pt        
     #1\endgraf}\hskip7truemm}}}\vss}}}
     
     \title{On the block number of graphs}
     \author{Daniel Wei\ss auer}
     \date{}
     
\begin{document}
	
	\maketitle     
     
     	\begin{abstract}
		A \emph{$k$-block} in a graph~$G$ is a maximal set of at least~$k$ vertices no two of which can be separated in~$G$ by deleting fewer than~$k$ vertices. The \emph{block number}~$\beta(G)$ of~$G$ is the maximum integer~$k$ for which~$G$ contains a $k$-block.
		
		We prove a structure theorem for graphs without a $(k+1)$-block, showing that every such graph has a \td\ in which every torso has at most~$k$ vertices of degree~$2k^2$ or greater. This yields a qualitative duality, since every graph that admits such a decomposition has block number at most~$2k^2$.
		
		We also study $k$-blocks in graphs from classes of graphs~$\G$ that exclude some fixed graph as a topological minor, and prove that every $G \in \G$ satisfies $\beta(G) \leq c\sqrt[3]{|G|}$ for some constant $c = c( \G)$. 
		
		 Moreover, we show that every graph of tree-width at least~$2k^2$ has a minor containing a $k$-block. This bound is best possible up to a multiplicative constant. 
\end{abstract}	
	
	\begin{section}{Introduction}
			Given $k \in \mathbb{N}$, a set~$X$ of at least~$k$ vertices of a graph~$G$ is \emph{$(<\! k)$-inseparable} if no two vertices in~$X$ can be separated in~$G$ by deleting fewer than~$k$ vertices. A maximal such set is a \emph{$k$-block} and can be thought of as a highly connected part of the graph, although it may draw its connectivity from the ambient graph~$G$ rather than just the subgraph induced by~$X$ itself. The maximum integer~$k$ for which~$G$ contains a $k$-block is the \emph{block number} of~$G$, denoted by~$\beta(G)$.
						
		The notion of $k$-blocks is a successful concept in the theory of graph-decom\-positions. Carmesin, Diestel, Hundertmark and Stein~\cite{canondec} showed that $k$-blocks provide a natural model of a ``highly connected substructure'' into which a graph can be decomposed in a tree-like manner. This was further refined by Carmesin, Diestel, Hamann and Hundertmark~\cite{canondec1, canondec2} and by Carmesin and Gollin~\cite{pascal}. Following a question raised in~\cite{canondec}, the study of  graphs which do not contain $k$-blocks was initiated by Carmesin, Diestel, Hamann and Hundertmark in~\cite{forceblock}, with a focus on degree-conditions. Here, our emphasis lies on the structure of these graphs and we relate the block number to other width-parameters for graphs.

		Dualities between the occurrence of some highly connected substructure and a tree-like structure of the whole graph, such as between blockages and path-decompositions~\cite{pwdual}, brambles and \td s~\cite{twdual} or tangles and branch-decompositions~\cite{bwdual}, are of particular interest in structural graph theory. 
		
		A unified framework for duality theorems for width-parameters in graphs and matroids was developed by Diestel and Oum~\cite{duality1, duality2}. Based on this framework, Diestel, Eberenz and Erde~\cite{blockdual} proved a duality theorem for $k$-blocks and described a class~$\mathcal{T}_k$ of \td s such that a graph has no $k$-block if and only if it has a \td\ in~$\mathcal{T}_k$. The only downside is that~$\mathcal{T}_k$ is given rather abstractly and thus seems difficult to work with. 
		
		Here, we give a simpler class of \td s that still acts as an obstruction to the existence of a $k$-block, at the expense of a precise duality: We obtain a qualitative duality theorem with a numerical trade-off.

		\begin{theorem} \label{structure theorem}
			Let~$G$ be a graph and $k \geq 2$ an integer.
			
			\begin{enumerate}[(i)]
				\item If~$G$ has no $(k+1)$-block, then~$G$ has a \td\ in which every torso has at most~$k$ vertices of degree at least~$2k(k-1)$. Moreover, there is such a \td\ of adhesion less than~$k$.
				
				\item If~$G$ has a \td\ in which every torso has at most~$k$ vertices of degree at least~$k$, then~$G$ has no $(k+1)$-block.
			\end{enumerate}
			
		\end{theorem}
		
		This yields a qualitative duality: Every graph either has a $(k+1)$-block or a \td\ that demonstrates that it has no $2k^2$-block.

			We also study the block number of graphs in classes of graphs that do not contain some fixed graph as a topological minor. Dvo\v{r}\'{a}k~\cite{dvorak} implicitly characterized those classes~$\G$ for which there exists an upper bound on the block number of graphs in~$\G$. We shall make this characterization explicit in Section~\ref{minor-closed}.
			
			The absence of an absolute bound does not have to be the end of the story, however. For instance, while the tree-width of planar graphs cannot be bounded by a constant, the seminal Planar Separator Theorem of Lipton and Tarjan~\cite{liptontarjan} implies that $n$-vertex planar graphs have tree-width at most~$c \sqrt{n}$ for some constant $c > 0$. We prove a bound on the block number in the same spirit. Note that the Planar Separator Theorem can be extended to arbitrary minor-closed classes of graphs, as shown by Alon, Seymour and Thomas~\cite{AST}, but not to classes excluding a topological minor.
		
		\begin{theorem} \label{planar block theorem}
			Let~$\G$ be a class of graphs excluding some fixed graph as a topological minor. There exists a constant $c = c(\G)$ such that every $G \in \G$ satisfies $\beta(G) \leq c \sqrt[3]{|G|}$.
		\end{theorem}
		
		In fact, our proof of Theorem~\ref{planar block theorem} works with a slightly weaker notion of a highly connected substructure which we call a \emph{$k$-fan-set}: a set $X \sub V(G)$ for which every $x \in X$ has a set of~$k$ otherwise disjoint paths to~$X$. The maximum~$k$ for which~$G$ contains a $k$-fan-set is called the \mbox{$\infty$-\emph{admissibility}} of~$G$ (sometimes also \mbox{\emph{$\infty$-degeneracy}}~\cite{richerby}), denoted by~$\adm_{\infty}(G)$. It turns out that $k$-blocks and $k$-fan-sets are essentially interchangeable concepts.
	
	\begin{theorem} \label{fans and blocks theorem}
			For every graph~$G$
			\[
				\lfloor \frac{\adm_{\infty}(G) + 1}{2} \rfloor \leq  \beta(G) \leq \adm_{\infty}(G) .
			\]
			\end{theorem}
			
			
			We then study the relation between block number and tree-width. It is easy to see that $\beta(G) \leq \tw(G) + 1$, so the existence of a $k$-block forces large tree-width. However, a graph can have arbitrarily large tree-width and yet have no 5-block:  $k \times k$-grids are such graphs. Since tree-width does not increase when taking minors, the tree-width of~$G$ (plus one) is even an upper bound for the block number of \emph{every minor} of~$G$. We can prove a converse to this statement, namely that a graph with large tree-width must have a minor with large block number.

		\begin{theorem}  \label{tree-width block theorem}
			Let $k \geq 1$ be an integer and~$G$ a graph. If $\tw(G) \geq 2k^2 - 2$, then some minor of~$G$ contains a $k$-block. This bound is optimal up to a constant factor.
		\end{theorem}
			
		This paper is organized as follows. Section~\ref{prelims} contains a brief account of definitions, basic facts and terminology used in the rest of the paper. Theorem~\ref{structure theorem} will be proven in Section~\ref{duality}. In Section~\ref{minor-closed} we prove a strong form of Theorem~\ref{planar block theorem} about $k$-blocks in classes of graphs excluding a topological minor. Theorem~\ref{fans and blocks theorem} will be proven in Section~\ref{admissibility}. A more precise version of Theorem~\ref{tree-width block theorem}, relating tree-width to the occurrence of $k$-blocks in a minor, will be proven in Section~\ref{tree-width}. Section~\ref{conclusion} contains some remarks on how our results fit into and extend the existing body of research as well as some open problems.
		
		\end{section}
		
		\begin{section}{Preliminaries} \label{prelims}
		All graphs considered here are finite and undirected, contain neither loops nor parallel edges and will be written as $G = (V, E)$. Our notation and terminology mostly follow that of~\cite{thebook}. Any graph-theoretic terms not defined here are explained there.
		
		For $X \sub V(G)$, an \emph{$X$-path} is a path of length at least one which meets~$X$ precisely in its endvertices. For a vertex $v \in V$ and integer~$k$, a \emph{$k$-fan from~$v$} is a collection~$\Q$ of~$k$ paths which all have~$v$ as a common starting vertex and are otherwise disjoint. It is a $k$-fan \emph{to} some $U \sub V$ if the end-vertex of every path in~$\Q$ lies in~$U$. We explicitly allow a fan to contain the trivial path consisting only of the vertex~$v$ itself. The end-vertex of this path is~$v$. A set $X \sub V$ is a \emph{$k$-fan-set} if from every $x \in X$ there exists a $k$-fan to~$X$. The \mbox{\emph{$\infty$-admissibility}} $\adm_{\infty}(G)$ is the maximum~$k$ for which~$G$ contains a $k$-fan-set.
		

			\begin{lemma} \label{block to fan}
			Let~$G$ be a graph and $X \sub V(G)$ a $(<\! k)$-inseparable set of vertices for some $k \in \mathbb{N}$. Then~$X$ is a $k$-fan-set.
		\end{lemma}
		
		\begin{proof}
		Suppose there was an $x \in X$ with no $(k-1)$-fan from~$x$ to $X \setminus \{ x \}$. By Menger's Theorem there is a set $S \sub V \setminus \{ x \}$ with $|S| < k-1$ separating~$x$ from $X \setminus ( S \cup \{ x \} )$. Since $|X| \geq k$, there must be some $y \in X \setminus (S \cup \{ x \})$. Since~$X$ is $(<\! k)$-inseparable, $S$ cannot separate~$x$ and~$y$, a contradiction.
		\end{proof}		
		
		The converse is not true: If~$G$ is a disjoint union of cliques of order~$k$, then $V(G)$ is a $k$-fan-set, but not even a 1-block.
		
		 If $X \sub V$ is $(<\! k)$-inseparable, then every $x \in X$ has degree at least $k-1$ in~$G$. Therefore~$G$ must have at least $k(k-1)/2$ edges. Since any minor of~$G$ has at most $e(G)$ edges, it follows that 
				\begin{equation} 
		\label{block num edges}
				\max_{H \prec G} \, \beta(H) \leq 1 + \sqrt{2e(G)} .
		\end{equation}
		
		If~$T$ is a tree and $s, t \in V(T)$, we denote by $sTt$ the unique path in~$T$ from~$s$ to~$t$. Recall that a \emph{\td} of $G$ is a pair $(T,\V)$ of a tree $T$ and a family $\V=(V_t)_{t\in T}$ of vertex sets $V_t\sub V(G)$, one for every node of~$T$, such that:
\begin{enumerate}[(T1)]
\item $V(G) = \bigcup_{t\in T}V_t$,
\item for every edge $uv \in E(G)$ there exists a $t \in T$ with $\{ u, v \} \sub V_t$,
\item $V_{t_1} \cap V_{t_3} \sub V_{t_2}$ whenever $t_2 \in t_1Tt_3$.
\end{enumerate}

The sets $V_t$, $t \in T$, in a \td\ are its \emph{parts}, while the sets $V_s \cap V_t$, $st \in E(T)$, are its \emph{adhesion-sets}. For $t \in T$ the \emph{torso of~$t$} is the graph obtained from $G[V_t]$ by adding edges between any two vertices of~$V_t$ that lie in a common adhesion-set. 

The \emph{adhesion} of~$(T, \V)$ is the maximum size of an adhesion-set. The \emph{width} of $(T,\V)$ is $\max_{t\in T}(|V_t|-1)$ and the \emph{tree-width} $\tw(G)$ of~$G$ is the minimum width of any of its \td s.

	\end{section}

	\begin{section}{The structure of graphs without $k$-blocks}
		\label{duality}

	Perhaps the most trivial reason a graph~$G$ can fail to contain a $(k+1)$-block is if~$G$ has at most~$k$ vertices of degree at least~$k$. These graphs can be used as building blocks for graphs of block number at most~$k$
	
	\begin{proof}[Proof of Theorem~\ref{structure theorem}~(ii)]
		Let $(T, \V)$ be a \td\ in which every torso has at most~$k$ vertices of degree at least~$k$. Assume that~$G$ contained a $(k+1)$-block~$X$. Every adhesion-set $V_s \cap V_t$, $st \in E(T)$, is a clique in the torso of~$t$, so $|V_s \cap V_t| \leq k$ by assumption on the degrees. 
		
		Since~$X$ is a $(k+1)$-block, it follows from a standard technique that there is a $t \in T$ with $X \sub V_t$, see \cite[Lemma~12.3.4]{thebook}. We will show that every vertex of~$X$ has degree at least~$k$ in the torso of~$t$, which is a contradiction.
	
	Let $x \in X$ arbitrary and let $A \sub V_t$ be the set of all neighbors of~$x$ in the torso of~$t$. If $|A| \geq k$, we are done. Otherwise, let $y \in X \setminus (A \cup \{ x \})$. In particular, $x$ and~$y$ are non-adjacent in~$G$, so by Menger's Theorem there is a set~$\Po$ of~$k+1$ internally disjoint $x$-$y$-paths in~$G$. Since every $P \in \Po$ has both end-vertices in~$V_t$, it has a vertex $z_P \in V(P) \cap V_t \setminus \{ x \}$ which lies closest to~$x$ along~$P$. Then~$x$ and this vertex~$z_P$ must either be adjacent or lie in a common adhesion-set $V_s \cap V_t$. Hence $z_P \in A$ and, in particular, $z_P \neq y$. As the paths in~$\Po$ are internally disjoint, all these vertices~$z_P$ are distinct. Thus the degree of~$x$ in the torso of~$t$ is at least~$k+1$.
	\end{proof}

		The converse, decomposing a graph with no $(k+1)$-block into graphs of almost bounded degree, is more intricate. 
		
		The \emph{fatness} of a \td\ $(T, \V)$ of an $n$-vertex graph~$G$ is the $(n+1)$-tuple $(a_0, \ldots , a_n)$ where~$a_i$ denotes the number of parts of $(T, \V)$ of size~$n-i$. If $(T, \V)$ has lexicographically minimum fatness among all \td s of adhesion less than~$k$, we call $(T, \V)$ \emph{$k$-atomic}. We are going to show that the high-degree vertices of a torso of a $k$-atomic \td\ cannot be separated by deleting fewer than~$k$ vertices.
		
		Every edge $t_1t_2 \in E(T)$ yields a separation of~$G$ as follows. Let $T_1,T_2$ be the two components of $T - t_1t_2$, where $t_1 \in T_1$, and let $G_i := \bigcup_{t \in T_i} V_t$ for $i \in \{ 1, 2\}$. Then $(G_1, G_2)$ is a separation of~$G$ with separator $X := G_1 \cap G_2 = V_{t_1} \cap V_{t_2}$.
			
	\begin{lemma} \label{atomic stable}
		Let~$(T, \V)$ a $k$-atomic \td\ of~$G$. For any $t_1t_2 \in E(T)$, there is a component~$C$ of $G_1 - X$ such that every $x \in X$ has a neighbor in~$C$.
	\end{lemma}
	
	\begin{proof}
		Suppose this was not the case. Let $C_1, \ldots, C_m$ be the components of $G_1 - X$. Obtain the tree~$T'$ from the disjoint union of~$m$ copies $T_1^1, \ldots, T_1^m$ of~$T_1$, where each $t \in T_1$ corresponds to~$m$ vertices $t^i \in T_1^i$ for $i \in [m]$, and one copy of~$T_2$ by joining~$t_2$ to every~$t_1^i$, $i \in [m]$. For $t \in T_2$ let $V_t' = V_t$, for $t \in T_1$ let $V_{t^i}' = V_t \cap (C_i \cup N(C_i))$. Observe that the adhesion of~$(T', \V')$ is less than~$k$. Let the fatness of $(T, \V)$ be $a = (a_i)_i$ and let the fatness of $(T', \V')$ be $a' = (a'_i)_i$.
		
		Clearly $|V_{t^i}'| \leq |V_t|$ for every $t \in T_1$. If $|V_{t^i}'| = |V_t|$ for some $i \in [m]$, then $V_t \sub C_i \cup N(C_i)$ and for all $j \neq i$ we have $|V_{t^j}'| \leq |N(C_i)| < |X|$. Choose $t \in T_1$ with $r := |G| - |V_t|$ minimum under the condition that there is no~$i$ with $|V_{t^i}| = |V_t|$. Since $N(C_i) \subsetneq X$ for every $i \in [m]$, the node $t_1$ satisfies this condition. Thus $r \leq |G| - |V_{t_1}| \leq |G| - |X|$. Then $a_s = a'_s$ for all $s < r$ and $a_r > a'_r$, so that~$a'$ is lexicographically smaller than~$a$, a contradiction. 
	\end{proof}
			
	This lemma helps us use the assumption of large degree in a torso.

	\begin{lemma} \label{torso degree to fan}
		Let~$(T, \V)$ a $k$-atomic \td\ of~$G$ for $k \geq 3$. Let $m \geq 1$ be an integer, $t \in T$ and $x \in V_t$. If~$x$ has degree at least $(m-1)(k-2)$ in the torso of~$t$, then there exists an $m$-fan from~$x$ to~$V_t$.
	\end{lemma}
				
			\begin{proof}
	Let~$A$ be the set of vertices of~$V_t$ which are adjacent to~$x$ in~$G$ and let~$B$ be the set of vertices that are adjacent to~$x$ in the torso of~$t$, but not in~$G$. Let~$\Q_A$ be the fan consisting of the trivial path~$\{ x \}$ and single edges to each $a \in A$. We now construct a fan~$\Q_B$ from~$x$ to~$B$ consisting of $V_t$-paths. 
			
		For every $b \in B$ there is an edge $st \in E(T)$ with $\{ b, x \} \sub V_s \cap V_t$. Let~$R$ be the set of all neighbors~$s$ of~$t$ in~$T$ with $x \in V_s$. Let $S \sub R$ minimal such that $B \sub \bigcup_{s \in S} V_s$.  For $s \in S$, let $B_s := B \cap V_s$. By minimality of~$S$, every~$B_s$ contains some vertex $b_s \notin \bigcup_{s' \neq s} B_{s'}$. 
		
		Let~$T_s$ be the component of $T-st$ containing~$s$ and $G_s := \bigcup_{r \in T_s} V_r$. By Lemma~\ref{atomic stable}, there is a component of $G_s - V_t$ that contains neighbors of both~$x$ and~$b_s$. We therefore find a path~$P_s$ from~$x$ to~$b_s$ in~$G_s$ that meets~$V_t$ only in its endpoints. The set $\Q_B := \{ P_s \colon s \in S \}$ is an $|S|$-fan from~$x$ to~$B$ in which every path is internally disjoint from~$V_t$. Thus $\Q := \Q_A \cup \Q_B$ is a fan from~$x$ to~$V_t$.
		
		It remains to show $|\Q| \geq m$. As $B \sub \bigcup_s B_s$, we have
		\[
			|A| + \sum_{s \in S} |B_s| \geq |A| + |B| \geq (m-1)(k-2) .
		\]
		Note that $x \in V_s \cap V_t$ for every $s \in S$ and $B_s \sub V_s \cap V_t \setminus \{ x \}$. Since~$(T, \V)$ has adhesion less than~$k$, it follows that $|B_s| \leq k-2$. Therefore
		\[
		|\Q| = 1 + |A| + |S| \geq 1 + \frac{|A| + (k-2)|S|}{k-2} \geq m .
		\]
	\end{proof}
				
		A \td\ is \emph{$k$-lean} if it has adhesion less than~$k$ and for any $s, t \in T$, not necessarily distinct, and any $A \sub V_s$, $B \sub V_t$ with $|A| = |B| \leq k$ either there is a set of~$|A|$ disjoint $A$-$B$-paths in~$G$ or there is an edge $uw \in E(sTt)$ with $|V_u \cap V_w| < |A|$. As observed in~\cite{forceblock}, the short proof of Thomas' theorem~\cite{thomas} given in~\cite{bellenbaum} in fact shows the following.
				
	\begin{theorem}[\cite{bellenbaum}] \label{atomic lean}
			Every $k$-atomic \td\ is $k$-lean.
		\end{theorem}
		
			\begin{lemma} \label{lean treedec fan block}
		Let~$(T, \V)$ be a $k$-lean tree-decomposition of~$G$, $t \in T$ and $u, v \in V(G)$. If from both~$u$ and~$v$ there are $(2k-1)$-fans to~$V_t$, then~$u$ and~$v$ cannot be separated by deleting fewer than~$k$ vertices.
	\end{lemma}			
	
	\begin{proof}
		Suppose there was some $S \sub V(G) \setminus \{ u, v \}$, $|S| < k$, separating~$u$ and~$v$. We find a set of~$k$ paths of the fan from~$u$ to~$V_t$ which are disjoint from~$S$ and let $R_u \sub V_t$ be their endvertices. Note that all vertices in~$R_u$ lie in the component of $G - S$ containing~$u$. Define $R_v \sub V_t$ similarly for~$v$. 
		
		Since $(T, \V)$ is $k$-lean, we find~$k$ vertex-disjoint paths from~$R_u$ to~$R_v $. All of these paths must pass through~$S$, a contradiction.
	\end{proof}
	
	We now combine all this to complete the proof of Theorem~\ref{structure theorem}.
	
		\begin{proof}[Proof of Theorem~\ref{structure theorem}~(i)]
	Let~$(T, \V)$ be a $(k+1)$-atomic \td\ of~$G$. For $t \in T$ let $X_t \sub V_t$ be the set of vertices of degree at least $2k(k-1)$ in the torso of~$t$. By Lemma~\ref{torso degree to fan}, every $x \in X_t$ has a $(2k+1)$-fan to~$V_t$. By Lemma~\ref{lean treedec fan block}, no two vertices of~$X_t$ can be separated by deleting fewer than~$k+1$ vertices. Since~$G$ has no $(k+1)$-block, it follows that $|X_t| \leq  k$.
	\end{proof}

	\end{section}

			\begin{section}{Excluded topological minors and $k$-blocks}
		\label{minor-closed}
		
		When considering $k$-blocks, the topological minor relation is more natural than the ordinary minor relation. For example, it is easy to see that a $(k+1)$-block in a graph~$H$ yields a $(< \! k)$-inseparable set in any graph containing~$H$ as a topological minor. No such statement is true when considering minors: It is easy to construct a triangle-free graph~$G$ of maximum degree~3 that contains the complete graph of order~$k$ as a minor. This graph~$G$ has no 4-block.

		In this section we study the block number of graphs from classes of graphs~$\G$ that exclude some fixed graph as a topological minor. Examples of such classes include graphs of bounded genus, bounded tree-width or bounded degree. In general, there exists no upper bound on the block number of graphs in~$\G$. In fact, we can explicitly describe a planar graph with block number~$k$: take a rectangular $rk \times k$-grid, add $2(r+1)$ vertices to the outer face and join each of these to~$k$ vertices on the perimeter of the grid (see Figure~\ref{bestplanar}). If $2(r+1) \geq k$, these new vertices are $(<\! k)$-inseparable.
		
		\begin{figure}[ht]
			\begin{center}
			\includegraphics[width=6cm]{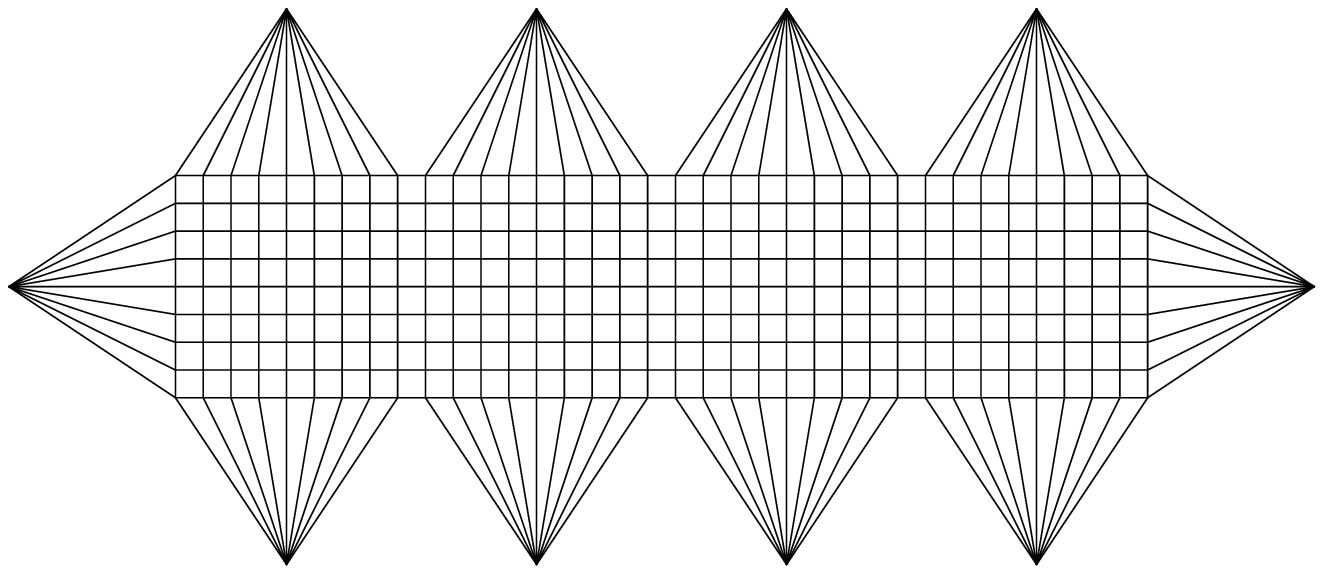}
			\end{center}
			\caption{A planar graph with a 9-block of order~10.}
			\label{bestplanar}
		\end{figure}
		
		We are thus faced with two tasks: First, to characterize those classes for which there exists an upper bound on the block number. Second, to obtain a relative upper bound on the block number of graphs in~$\G$ when no absolute upper bound exists.

%
%

	\begin{subsection}{The bounded case}
			As indicated in the introduction, Dvo\v{r}\'{a}k~\cite{dvorak} implicitly characterized the classes for which there exists an upper bound on the block number. Since $k$-blocks are not mentioned in~\cite{dvorak}, we make this characterization explicit here without adding any ideas not present in~\cite{dvorak}.

	A small modification of the graph depicted in Figure~\ref{bestplanar} yields a planar graph~$H_k$ with roughly~$k^3/2$ vertices and block number~$k$ which can be drawn in the plane so that every vertex of degree greater than~3 lies on the outer face: Essentially, replace the square grid by a hexagonal grid and join the `new' vertices only to degree-2 vertices on the perimeter.			
			
			Suppose that~$H$ is a graph with the property that every graph~$G$ that does not contain~$H$ as a topological minor satisfies $\beta(G) < s$ for some constant $s = s(H)$. Then~$H$ is a topological minor of~$H_s$ and therefore planar. Moreover, $H$ ``inherits'' a drawing in the plane in which all vertices of degree greater than~3 lie on the outer face. 
			
			The simplest case of a deep structure theorem for graphs excluding a fixed graph as a topological minor~\cite[Theorem~3]{dvorak} asserts a converse to this in a strong form.

%
%
						
					\begin{theorem}[\cite{dvorak}] \label{dvorak theorem}
				Let~$H$ be a graph drawn in the plane so that every vertex of degree greater than~3 lies on the outer face. Then there exists an $r = r(H)$ such that every graph that does not contain~$H$ as a topological minor has a \td\ in which every torso contains at most~$r$ vertices of degree at least~$r$.			
			\end{theorem}
			
	It is now easy to characterize the graphs whose exclusion as a topological minor bounds the block number.
			
			\begin{corollary}
					Let~$H$ be a graph. The following are equivalent:
					\begin{enumerate}[(i)]
						\item There is an integer $s = s(H)$ such that every graph~$G$ that does not contain~$H$ as a topological minor satisfies $\beta(G) < s$.
						\item $H$ can be drawn in the plane such that
						 every vertex of degree greater than~3 lies on the outer face.
					\end{enumerate}
				\end{corollary}
						
					\begin{proof}
				(i) $\rightarrow$ (ii): By assumption, the graph~$H_s$ contains~$H$ as a topological minor. The desired drawing of~$H$ can then be obtained from the drawing of~$H_s$.
				
				(ii) $\rightarrow$ (i): By Theorem~\ref{dvorak theorem} and Theorem~\ref{structure theorem}~(ii).
				\end{proof}

	Note that every graph that contains~$H_k$ as a topological minor necessarily has a $k$-block. Theorem~\ref{dvorak theorem} thereby implies a qualitative version of Theorem~\ref{structure theorem}~(i), but without explicit bounds.			
				
				\begin{corollary} \label{class bounded char}
					Let~$\G$ be a class of graphs. The following are equivalent:
\begin{enumerate}[(1)]
	\item There is a $k \in \mathbb{N}$ such that $\beta(G) \leq k$ for every $G \in \G$.
	\item There is an $m \in \mathbb{N}$ such that no $G \in \G$ contains~$H_m$ as a topological minor.
	\item There is an $r \in \mathbb{N}$ such that every graph in~$\G$ has a \td\ in which every torso has at most~$r$ vertices of degree at least~$r$.
\end{enumerate}				
				\end{corollary}

				\end{subsection}

	\begin{subsection}{The unbounded case}
	We now turn to the case where~$\G$ is a class of graphs excluding some fixed graph as a topological minor for which there exists no upper bound on the block number of graphs in~$\G$. If~$\G$ is closed under taking topological minors, then by Corollary~\ref{class bounded char} this implies $H_k \in \G$ for all $k \in \mathbb{N}$. Since $|H_k| \leq \beta(H_k)^3$, the bound in Theorem~\ref{planar block theorem} is optimal up to a constant factor.
	
	Our aim now is to prove Theorem~\ref{planar block theorem}. In light of Lemma~\ref{block to fan}, it clearly suffices to show the following.

		\begin{theorem} \label{planar block theorem strong}
			Let~$\G$ be a class of graphs excluding some fixed graph as a topological minor. There exists a constant $c = c(\G)$ such that every $G \in \G$ containing a $k$-fan-set~$X$ has at least $c|X|k^2$ vertices.
		\end{theorem}
		
	This immediately yields the following strengthening of Theorem~\ref{planar block theorem}.
	
		\begin{corollary}	
	Let~$\G$ be a class of graphs excluding some fixed graph as a topological minor. Let $G \in \G$ and let~$X$ be the set of all vertices of~$G$ that lie in some $k$-block of~$G$. Then $|X| \leq |G|/(ck^2)$, where $c = c(\G)$ is the constant from Theorem~\ref{planar block theorem strong}.
		\end{corollary}
		
		\begin{proof}
		By Lemma~\ref{block to fan}, every $k$-block of~$G$ is a $k$-fan-set. It is easy to see that a union of $k$-fan-sets is again a $k$-fan-set. Since~$X$ is the union of all $k$-blocks, it is therefore a $k$-fan-set. By Theorem~\ref{planar block theorem strong} we have $|G| \geq c|X|k^2$ for $c = c(\G)$.
		\end{proof}
				
		We now turn to the proof of Theorem~\ref{planar block theorem strong} above. Excluding a topological minor ensures that our graph and all its topological minors are sparse. The following is well-known, see \cite[Chapter~7]{thebook}.
		
		\begin{lemma} \label{minor closed sparse}
		Let~$\G$ be a class of graphs excluding some fixed graph as a topological minor. There exist constants $\alpha, d > 0$ such that every topological minor~$G$ of a graph in~$\G$ has at most $d|G|$ edges and an independent set of order at least~$\alpha |G|$.
		\end{lemma}

		\begin{proof}[Proof of Theorem~\ref{planar block theorem strong}]
		Let $G \in \G$ and $k \in \mathbb{N}$. To ease notation, we assume that $X \sub V(G)$ is a $(k+1)$-fan-set instead of just a $k$-fan-set. This only has an effect on the constant~$c$.
		
		For every $x \in X$ let~$\Q_x$ be a $k$-fan from~$x$ to~$X \setminus \{ x \}$. Taking subpaths, if necessary, we may assume that no $Q \in \Q_x$ has an internal vertex in~$X$. We use initial segments of the paths in~$\Q_x$ to construct a subdivision of a star with center~$x$. Lemma~\ref{minor closed sparse} will enable us to find many disjoint such subgraphs.

		We adopt an idea from~\cite{grohe}. For some integer~$r$ that we are going to choose later, let~$\Po$ be a maximal set of internally disjoint $X$-paths of length at most~$2r$ such that for any two $x, y \in X$ there is at most one path in~$\Po$ joining them. The paths in~$\Po$ will be used as barriers to separate the subdivided stars. Let $B := X \cup \bigcup \Po$.
				
		For $x \in X$ and $Q \in \Q_x$, let $Q' \sub Q$ be the maximal subpath of length at most~$r$ with $Q' \cap B = \{ x \}$. If the length of~$Q'$ is less than~$r$, then the next vertex along~$Q$ lies in~$B$. We say that this vertex \emph{stops} the path~$Q'$. Define $\Q_x' := \{ Q' \colon Q \in \Q_x \}$ and $S_x := \bigcup \Q_x'$.
		
		The paths of~$\Po$ provide us control on the overlap of the stars and allow us to separate them. Let $H = H( \Po)$ be the auxiliary graph with vertex-set~$X$ where $xy \in E(H)$ if and only if some $P \in \Po$ joins~$x$ and~$y$. 
		\begin{equation} \label{barriers}
			\text{If } \bigcup \Q_x' \cap \bigcup \Q_y' \neq \emptyset, \text{ then } xy \in E(H) .
		\end{equation}
		Indeed, if $ \bigcup \Q_x' \cap \bigcup \Q_y' \neq \emptyset$ then we can find a path~$P$ of length at most~$2r$ between~$x$ and~$y$ which is internally disjoint from all paths in~$\Po$. By maximality of~$\Po$, there must already be some $R \in \Po$ joining~$x$ and~$y$. Similarly
		\begin{equation}
			\text{If } x \in X \text{ stops some } Q' \in \Q_y', \text{ then } xy \in E(H).
		\end{equation}
		
		The graph~$H$ is clearly a topological minor of~$G$. It follows from Lemma~\ref{minor closed sparse} that $|\Po| = |E(H)| \leq d|X|$ and that~$H$ contains an independent set $Y \sub X$ with $|Y| \geq \alpha |X|$. By~(\ref{barriers}), the stars with centers in~$Y$ are pairwise disjoint and $z \in Y$ does not stop any $Q' \in \Q_y'$ for $y \in Y$. We will show that, on average, many paths in~$\Q_y'$, $y \in Y$, have length~$r$. 
		
		For $y \in Y$ let~$q_y$ be the number of $Q' \in \Q_y'$ that were stopped. Extending each $Q' \in \Q_y'$ that was stopped by a single edge, we obtain a path~$Q''$ from~$y$ to the vertex $v \in B$ that stopped~$Q'$. Note that if~$v$ stops some $Q' \in \Q_y'$, then $v \in B \setminus Y$. We therefore obtain a bipartite graph~$J$ with $V(J) = Y \cup (B \setminus Y)$ as a topological minor of~$G$, where $yv \in E(J)$ if and only if~$v$ stops some $Q' \in \Q_y'$. It follows from Lemma~\ref{minor closed sparse} that
		\[
		e(J) \leq d|J| \leq d(|X| + (2r-1)|\Po|) \leq 2rd^2|X| . 
		\]
		Since the paths in~$\Q_y$ intersect only in~$y$, no vertex can stop more than one $Q' \in \Q_y'$. Therefore
		\[
		\sum_{y \in Y} q_y \leq  e(J) \leq 2rd^2 |X| .
		\]
	It follows that 		 
	\begin{align*}	
		|G| &\geq \sum_{y \in Y} | S_y| > \sum_{y \in Y} r(k- q_y) \geq r|Y|k - 2r^2d^2 |X| \\
		&\geq r|X| ( \alpha k - 2rd^2 ) .
	\end{align*}
	Setting $r := \lfloor \frac{\alpha k}{4d^2} \rfloor$ yields the desired result.	
\end{proof}			

	\end{subsection}

		\end{section}

		\begin{section}{Admissibility and $k$-blocks}
		\label{admissibility}
		
	We now prove Theorem~\ref{fans and blocks theorem}, which asserts that block number and \infadm\ are within a constant multiplicative factor. By Lemma~\ref{block to fan}, every $k$-block is a $k$-fan-set and so
	\[
	\beta(G) \leq \adm_{\infty}(G) .
	\]
	It thus only remains to show $\beta(G) \geq \lfloor (\adm_\infty(G) + 1) / 2 \rfloor$. Lemma~\ref{lean treedec fan block} provides a sufficient condition for a set of vertices to be $(<\! k)$-inseparable. Our proof is an adaptation of the proof of~\cite[Theorem~4.2]{forceblock}, where it is shown that $\beta(G) \geq \lfloor \delta(G) / 2 \rfloor + 1$. This is also a consequence of our result, since~$V(G)$ itself is a $(\delta(G) + 1)$-fan-set.

			\begin{proof}[Proof of Theorem~\ref{fans and blocks theorem}]
				The inequality $\beta(G) \leq \adm_{\infty}(G)$ follows from Lemma~\ref{block to fan}.
								
				Suppose now that $\adm_{\infty}(G) \geq 2k - 1$ and let $X \sub V(G)$ be a $(2k-1)$-fan-set. We will show that~$X$ contains a $(<\! k)$-inseparable set.
				
			By Lemma~\ref{atomic lean} there exists a $k$-lean \td\ $(T, \V)$ of~$G$. Let $S \sub T$ be a minimal subtree such that $X \sub \bigcup_{s \in S} V_s$. Let $t \in S$ be a leaf of~$S$. If $S = \{ t \}$, then $X \sub V_t$ and from every $x \in X$ there is a $(2k-1)$-fan to~$V_t$. By Lemma~\ref{lean treedec fan block}, $X$ itself is already $(<\! k)$-inseparable.
				
				Otherwise, let~$t'$ be the unique neighbor of~$t$ in~$S$ and let $W := X \cap V_t \setminus V_{t'}$. Note that $W \neq \emptyset$, for otherwise $S - t$ would violate the minimality of~$S$. Let $w \in W$ arbitrary and let~$\Q_w$ be a $(2k-1)$-fan from~$w$ to~$X$. Every $Q \in \Q_w$ whose endvertex is not in~$W$ must meet $V_t \cap V_{t'}$. Thus at most $|V_t \cap V_{t'}| < k$ paths from~$\Q_w$ have endvertices outside~$W$. In particular, $|W| \geq k$. Furthermore by stopping every $Q \in \Q_w$ when it hits $V_t \cap V_{t'}$ (if it does) we obtain a $(2k-1)$-fan from~$w$ to~$V_t$. By Lemma~\ref{lean treedec fan block} the vertices of~$W$ cannot be separated by deleting fewer than~$k$ vertices.
			\end{proof}

		\end{section}

				\begin{section}{Tree-width and $k$-blocks}
		\label{tree-width}
		This section is devoted to the relation between tree-width and the occurrence of $k$-blocks in a minor. By considering random graphs, one can show that there are graphs~$G_n$ on~$n$ vertices with~$2n$ edges and tree-width at least $\gamma n$ for some absolute constant $\gamma > 0$ (see~\cite[Corollary~5.2]{kloks}). By~(\ref{block num edges}) we have $\beta(H) \leq 1 + \sqrt{4n}$ for every $H \prec G_n$. Hence the bound in Theorem~\ref{tree-width block theorem} is best possible up to constant factors.
		
		We now show that every graph of tree-width at least $2(k^2-1)$ has a minor with block number at least~$k$. In fact, this follows easily from a lemma in the proof of the Grid Minor Theorem given by Diestel, Jensen, Gorbunov and Thomassen~\cite{diestelgrid}. To state their result, we need to introduce some terminology.
		
		Let~$G$ be a graph. Call a set~$X$ of vertices \emph{externally $k$-linked in~$G$} if for any $Y, Z \sub X$, $|Y| = |Z| \leq k$, there are~$|Y|$ disjoint $X$-paths joining~$Y$ and~$Z$. A \emph{$k$-mesh of order~$m$} is a separation $(A, B)$ with $|A \cap B| = m$ such that $A \cap B$ is externally $k$-linked in $G[B] - E(A \cap B)$ and there is a tree $T \sub G[A]$ with $\Delta(T) \leq 3$ such that every vertex of $A \cap B$ lies in~$T$ and has degree at most~2 in~$T$.
		
		\begin{lemma}[{\cite[Lemma~4]{diestelgrid}} ] \label{diestel et al}
			Let~$G$ be a graph and $m \geq k \geq 1$ integers. If $\tw(G) \geq k + m - 1$, then~$G$ has a $k$-mesh of order~$m$.
		\end{lemma}
	
	\begin{lemma}
		Let $p \geq 0, k \geq 2$ be integers and let~$T$ be a tree with $\Delta(T) \leq 3$ and $X \sub V(T)$ a set of at least $(2p+1)(k-1)$ vertices of degree at most~2. Then there are disjoint subtrees $T_1, \ldots, T_p \sub T$ such that $|T_i \cap X| \geq k$ for every $i \in [p]$.
	\end{lemma}
	
	\begin{proof}
		By induction on~$p$. The case where $p \in \{ 0, 1 \}$ is trivial. In the inductive step, declare a leaf~$r$ of~$T$ as the root and thus introduce an order on~$T$. Choose $t \in T$ maximal in the tree-order such that~$\lfloor t \rfloor$, the subtree containing~$t$ and all its descendants, contains at least~$k$ vertices of~$X$. Note that $t \neq r$, since $|X| > k$. 
		
		If $t \in X$, then $|\lfloor t \rfloor \cap X| = k$ because~$t$ has only one successor~$s$ and $|\lfloor s \rfloor \cap X| < k$. If $t \notin X$, then similarly $|\lfloor t \rfloor \cap X| \leq 2(k-1)$. Let $S := T - \lfloor t \rfloor$ and note that $|S \cap X| \geq |X| - 2(k-1)$. By the inductive hypothesis applied to~$S$ and $S \cap X$ we find disjoint $S_1, \ldots, S_{p-1} \sub S$ with $|S_i \cap X| \geq k$ for all $i \in [p-1]$.
		For $1 \leq i < p$ let $T_i := S_i$ and put $T_{p} := \lfloor t \rfloor$. These subtrees of~$T$ are as desired.
	\end{proof}
	
	We thus obtain the following more precise version of Theorem~\ref{tree-width block theorem}.
	
	\begin{theorem}
		Let~$G$ be a graph and $p \geq k \geq 2$ integers. If the tree-width of~$G$ is at least $2(k-1)(p+1)$, then some minor of~$G$ contains a $(<\! k)$-inseparable independent set of size~$p$.
	\end{theorem}

	\begin{proof}
	Let $m := \tw(G) - k + 1$. By Lemma~\ref{diestel et al} above, $G$ has a $k$-mesh $(A, B)$ of order~$m$. Let $T \sub G[A]$ be the tree guaranteed by the definition. 
	
	 Since $m \geq (k-1)(2p + 1)$, we can apply the lemma above to find disjoint subtrees $T_1, \ldots, T_p \sub T$ such that each contains at least~$k$ vertices of $A \cap B$.

	Let $W := (B \setminus A) \cup \bigcup_{i = 1}^p V(T_i)$ and obtain~$H$ from~$G[W]$ by deleting all edges between~$T_i$ and~$T_j$ for $i \neq j$. 	Given $1 \leq i , j \leq p$, the graph~$H$ contains~$k$ disjoint paths between $T_i \cap (A \cap B)$ and $T_j \cap (A \cap B)$ with no internal vertices or edges in $A \cap B$, since $A \cap B$ is externally $k$-linked in $G[B] - E(A, B)$.

 Contracting each~$T_i$ to a single vertex thus yields the desired $(<\! k)$-inseparable independent set in a minor of~$H$ and thus of~$G$.
		\end{proof}

	Taking $p = k$ clearly yields Theorem~\ref{tree-width block theorem}.			
		
		\end{section}

				\begin{section}{Concluding remarks}
				\label{conclusion}
								
				From Theorem~\ref{tree-width block theorem} and Lemma~\ref{block to fan} we deduce the following.
				
				\begin{corollary}
					Let $k \geq 1$ be an integer. Every graph of tree-width at least $2k^2 - 2$ has a minor with \infadm\ at least~$k$.
				\end{corollary}

		Richerby and Thilikos~\cite{richerby} proved the existence of a function~$g$ such that graphs of tree-width at least~$g(k)$ have a minor with \infadm~$\geq k$. In their proof, $g(k)$ is the minimum~$N$ such that graphs of tree-width at least~$N$ have the $k^{3/2} \times k^{3/2}$-grid as a minor. The existence of such an~$N$ is the rather difficult \emph{Grid-Minor Theorem} of Robertson and Seymour~\cite{excludeplanar}. In comparison, our proof is short and simple: the only non-trivial step was a lemma from~\cite{diestelgrid}, whose proof is about a page long and in fact the first step in their proof of the Grid-Minor Theorem. Moreover, we have provided an explicit quadratic bound on~$g(k)$, while even the existence of a polynomial bound upper bound for~$N$ is a recent breakthrough-result of Chekuri and Chuzhoy~\cite{polygrid}. \\

		Dvo\v{r}\'{a}k proved that for every~$k$ there are integers~$m$ and~$d$ such that every graph with \infadm\ at most~$k$ has a \td\ in which every torso contains at most~$m$ vertices of degree at least~$d$ (\cite[Corollary~5]{dvorak}). The proof is based on a deep structure theorem for graphs excluding a topological minor~\cite[Theorem~3]{dvorak} and does not yield explicit bounds for~$m$ and~$d$. Combining Theorem~\ref{structure theorem} with Lemma~\ref{block to fan}, we obtain a much simpler proof that avoids the use of advanced graph minor theory and moreover provides explicit values for the parameters involved. 		
	
		\begin{corollary}
			Let $k \geq 1$ be an integer. If~$G$ has \infadm\ at most~$k$, then~$G$ has a \td\ of adhesion less than~$k$ in which every torso contains at most~$k$ vertices of degree at least~$2k(k-1)$.
		\end{corollary}
		
	It seems challenging to obtain stronger estimates: What is the minimum $N = N(k)$ such that every graph without a $k$-block has a \td\ in which every torso contains at most~$N$ vertices of degree at least~$N$? Can we always find a \td\ in which every torso has a bounded number of vertices of degree at least~$\alpha k$ for some constant $\alpha > 0 $? \\
		
		Admissibility of graphs has primarily been studied with a length-restriction imposed. We call the maximum length of a path in a fan~$\Q$ the \emph{radius} of~$\Q$. A $(k,r)$-fan-set is a set $X \sub V(G)$ such that from every $x \in X$ there is a $k$-fan of radius at most~$r$ to~$X$. The \emph{$r$-admissibility} $\adm_r(G)$ is the maximum~$k$ for which~$G$ has a $(k,r)$-fan-set. In particular 
		\[
		1 + \max_{H \sub G} \, \delta(H) = \adm_1(G) \leq \adm_2(G) \leq \ldots \leq \adm_{|G|}(G) = \adm_{\infty}(G) .
		\]
		Note that for every integer $r \geq 1$ trivially
	\begin{equation}
		\label{triv admiss}
	\adm_r(G) > \adm_{\infty}(G) - \frac{|G|}{r+1} ,
	\end{equation}
		since a fan cannot contain $|G|/(r+1)$ paths of length~$>r$.
		
		Grohe et al showed in~\cite{grohe} that for every class of graphs~$\G$ excluding a topological minor we have $\adm_r(G) = \mathcal{O}(r)$ for every $G \in \G$. Taking~(\ref{triv admiss}) into account we obtain the trivial estimate $\adm_{\infty}(G) = \mathcal{O}(\sqrt{|G|})$ for $G \in \G$, which also follows from a simple edge-count and Lemma~\ref{minor closed sparse}. On the other hand, Theorem~\ref{planar block theorem strong} shows that $\adm_r(G) = \mathcal{O}( \sqrt[3]{|G|})$ for every~$r$. Hence for values of~$r$ which are large with respect to~$|G|$, namely for $r \geq K \sqrt[3]{|G|}$ for some constant $K > 0$, our result is a substantial improvement of the estimate of Grohe et al.

		Let~$\G$ be class of graphs excluding a topological minor. For $n, r \in \mathbb{N}$ let
		\[
		F(n,r) := \max \{ \adm_{r}(G) \colon G \in \G, |G| = n \} .
		\]
	We know by Theorem~\ref{planar block theorem strong} that $F(n,r) = \mathcal{O}(\sqrt[3]{n})$ for all $r \in \mathbb{N}$, while Grohe et al~\cite{grohe} showed $F(n,r) = \mathcal{O}(r)$ for all $n \in \mathbb{N}$. It appears to be an interesting problem to try to obtain a unified bound. 
			\end{section}
		
		\section*{Acknowledgements}{
			
			I would like to thank the people of the graph-theory group at the University of Hamburg, in particular Joshua Erde who pointed out that $k$-lean tree-decompositions might help in a proof of Theorem~\ref{fans and blocks theorem}.
		}

 		\newpage 
	\bibliographystyle{plain}
\bibliography{blockbib}

     \end{document}